\documentclass[12pt,twoside]{article}
\usepackage[ansinew]{inputenc}
\usepackage{amsfonts}
\usepackage{latexsym}
\usepackage{amsmath}
\usepackage{amssymb}
\usepackage{eepic}
\usepackage{graphicx}
\usepackage{float}
\usepackage{pstricks}
\usepackage{epic,eepic}
 \usepackage{xcolor}
\usepackage{tikz}
\usetikzlibrary{arrows,shapes,chains}

\newcommand{\veps}{\varepsilon}
\newcommand{\R}{\mathbb{R}}

\newcommand{\C}{\mathbb{C}}
\newcommand{\N}{\mathbb{N}}

\newcommand{\A}{\mathcal{A}}
\renewcommand{\L}{\mathcal{L}}

\newcommand{\E}{\mathcal{E}}
\newcommand{\Ce}{\mathcal{C}}

\newcommand{\D}{\mathbb{D}}

\newtheorem{defin}{Definition}[section]

\newtheorem{theorem}[defin]{Theorem}

\newtheorem{exa}[defin]{Example}
\newenvironment{example}{\begin{exa}\rm}{\end{exa}}
\newtheorem{lemma}[defin]{Lemma}
\newtheorem{corollary}[defin]{Corollary}
\newenvironment{proof}
{\noindent{\it Proof.}}{\hfill $\Box$\par\vspace{2.5mm}}
\newenvironment{remark}
{\par\vspace{2.5mm}\noindent{\bf Remark.}}{\par\vspace{2.5mm}}

\newtheorem{que}{Question}

\newtheorem{pro}{Problem}

\numberwithin{equation}{section}

\makeatletter
\renewcommand{\ps@myheadings}{%
\renewcommand{\@evenhead}%
{{\rm\thepage}\hfil{\sc  J.~Ding, J.~Heittokangas, Z.-T.~Wen}\hfil}%
\renewcommand{\@oddhead}%
{\hfil{{\sc A continuous transition from $\E$-sets to $R$-sets}\hfil{\rm\thepage}}}%
\renewcommand{\@evenfoot}{}%
\renewcommand{\@oddfoot}{}%
}\makeatother 

\title{\bf\Large A continuous transition from $\E$-sets to $R$-sets and beyond}
\author{Jie Ding, Janne Heittokangas\footnote{Corresponding author.}$\, $
and Zhi-Tao Wen
   }
\date{}
\begin{document}
\maketitle

\begin{abstract}
The well-known $\E$-sets introduced by Hayman in 1960 are collections of Euclidean discs in the complex plane with the following property: The set of angles $\theta$ for which the ray $\arg(z)=\theta$ meets infinitely many discs in a given $\E$-set has linear measure zero. An important special case of an $\E$-set is known as the $R$-set. These sets appear in numerous papers in the theories of complex differential and functional equations. This paper offers a continuous transition from $\E$-sets to $R$-sets, and then to much thinner sets. In addition to rays, plane curves that originate from the zero distribution theory of exponential polynomials will be considered. It turns out that almost every such curve meets at most finitely many discs in the collection in question. Analogous discussions are provided in the case of the unit disc $\D$, where the curves tend to the boundary $\partial\D$ tangentially or non-tangentially. Finally, these findings will be used for improving well-known estimates for logarithmic derivatives, logarithmic differences and logarithmic $q$-differences of meromorphic functions, as well as for improving standard results on
exceptional sets. 


\end{abstract}

\renewcommand{\thefootnote}{}
\footnotetext[1]{Heittokangas was supported by the V\"ais\"al\"a Fund of the Finnish Academy of Science and Letters. Wen was supported by National Natural Science Foundation of China (No.~11971288) and Shantou University SRFT (NTF18029).}


\section{Introduction}

In 1960 Hayman \cite{Hayman} introduced an $\E$-set as a
countable collection of Euclidean discs $D(z_n,r_n)$ not including the origin, for which $z_n\in\C$, $|z_n|\to\infty$, $r_n>0$, and whose subtending angles at the origin have a finite sum. Any such disc
can be see from the origin at an angle $2\theta_n$,
where $\sin\theta_n=r_n/|z_n|$. Thus, by the definition,
	\begin{equation}\label{E-sum}
	\sum_{n} \frac{r_n}{|z_n|}<\infty.
	\end{equation}
The projection $E$ of such discs onto the interval $[1,\infty)$ has a finite logarithmic measure. This follows from
	\begin{eqnarray*}
	\int_E \frac{dx}{x} &\leq& \sum_{n}\int_{|z_n|-r_n}^{|z_n|+r_n}\frac{dx}{x}+O(1)
	\leq \sum_{n} \frac{2r_n}{|z_n|-r_n}+O(1)\\
	&=&\sum_{n}\frac{2r_n/|z_n|}{1-r_n/|z_n|}+O(1)<\infty,
	\end{eqnarray*}
because $r_n/|z_n|\to 0$ as $n\to\infty$ by \eqref{E-sum}.

A well-known special case of an $\E$-set is an $R$-set \cite[p.~84]{Laine}, which is a collection of Euclidean discs $D(\zeta_n,\rho_n)$, for
which $\zeta_n\in\C$, $|\zeta_n|\to\infty$, $\rho_n>0$, and whose
diameters have a finite sum. Thus, the projection $E$ of an $R$-set onto the positive real axis has a finite linear measure: $\int_E dr<\infty$.

If \eqref{E-sum} holds, then for any $\veps>0$ there exists an $r(\veps)>0$ such that
	$$
	\sum_{|z_n|>r(\veps)} \frac{r_n}{|z_n|}<\veps.
	$$
This gives raise to the fact that the set of angles $\theta$ for which the ray $\arg(z)=\theta$ meets infinitely many discs of a given $\E$-set has linear measure zero \cite{Hayman,Laine}. The same conclusion obviously holds for any $R$-set.

The next elementary example shows that infinitely many rays can meet
infinitely many discs of a given $R$-set.

\begin{example}\label{example1}
Let $\theta_n=1/n$ for $n\in\N$. For each $n\in\N$ choose any sequence of discs $D(z_{n,k},r_{n,k})$ such that $\arg(z_{n,k})=\theta_n$, $|z_{n,k}|\to\infty$ as $k\to\infty$, $r_{n,k}>0$, and that
	$$
	\sum_{k=1}^\infty r_{n,k}\leq 2^{-n}.
	$$
Then each ray $\arg(z)=\theta_n$ meets infinitely many discs and the sum of the diameters of the discs in the entire countable collection $\{D(z_{n,k},r_{n,k})\}$ is $\leq 2$.
\end{example}

More generally, we will prove that certain plane curves that either 
drift away or asymptotically approach to a given critical ray 
$\arg(z)=\phi$ avoid a certain collection of discs. When $\phi=0$, these curves take a simple form
	$$
	y=\pm cK(x)\quad\textnormal{or}\quad y=\pm cL(x),\quad c>0,
	$$
where $K$ is increasing, continuous, concave, and essentially satisfies $1\leq K(x)\leq x$, while $L$ is decreasing, continuous, convex and $L(x)\to 0^+$ as $x\to\infty$. Domains surrounded by curves of either type around finitely many symmetrically separated critical rays are known to contain the majority of zeros of exponential polynomials \cite{HITW, HW, Stein}. This property is, in fact, the motivation for the present paper.

The collections of discs associated with the aforementioned curves depend on the given curve type. In addition, the projection $E$ of these discs onto the interval $[1,\infty)$ turns out to be either
	$$
	\int_E \frac{dx}{K(x)}<\infty\quad\textnormal{or}\quad
	\int_E \frac{dx}{L(x)}<\infty.
	$$
The former gives a continuous transition from sets of finite logarithmic measure to sets of finite linear measure, while the latter transits the sets of finite linear measure to much thinner sets.

The discussions related to the previous two paragraphs will be carried out in detail in Section~\ref{plane-curves}. Analogous situations are then considered in the case of the unit disc $\D$ in Section~\ref{disc-curves}.
In particular, we will demonstrate that certain collections of Euclidean discs can be avoided with different families of curves that tend to the
boundary $\partial\D$ tangentially or non-tangentially.
A continuous transition of standard exceptional sets will also be given.

The idea of giving a continuous transition for exceptional sets from finite logarithmic measure to smaller sets is not entirely new neither in the case of $\C$ nor in the case of $\D$. Here, we wish to acknowledge the seminal work by Hinkkanen \cite{Hinkkanen} and Ye \cite{Ye}, in describing the size of the error term in Nevanlinna's second fundamental theorem.

Gundersen used $R$-sets and $\E$-sets (but under different teminology) in finding sharp estimates for logarithmic derivatives of meromorphic functions \cite{Gundersen}.
Applying these findings, Chiang and Feng obtained pointwise estimates for logarithmic differences of meromorphic functions \cite{C-F}, while Wen and Ye estimated logarithmic $q$-differences of meromorphic functions \cite{WY}. Estimates in these directions have had numerous applications in the theories of complex differential equations and complex difference equations. As an application of the continuous transition from sets of finite logarithmic measure to smaller sets, we are able to obtain refined pointwise estimates for logarithmic derivatives, logarithmic differences and logarithmic $q$-differences of meromorphic functions.
The details will be carried out in detail in Section~\ref{ld-ld}.

The existing literature contains lemmas to avoid an exceptional set of
finite linear measure or finite logarithmic measure. In Section~\ref{ld-ld} we use the continuous transition method in develping a lemma that allows us to avoid all exceptional sets that are at most of finite logarithmic measure in an optimal way. This approach improves the existing standard approaches.


\section{Results for curves in the plane}\label{plane-curves}

Our construction is based on concavity and convexity
of curves. The functions involved also have to satisfy a regularity
condition, which goes by the name ''doubling'' in the literature.

\subsection{Concave curves}\label{K}

Let $K:[0,\infty)\to [0,\infty)$ be a strictly increasing, continuous, concave function satisfying $K(x_0)=0$ for some $x_0\geq 0$. Note that the maximal growth rate for such a function $K$ is $K(x)=O(x)$. Moreover, we suppose that $K$ satisfies the following doubling condition: There exist constants $\alpha>1$ and $R\geq 1$ such that 
	\begin{equation}\label{double}
	K(2x) \leq \alpha K(x),\quad x\geq R.
	\end{equation}

The doubling condition \eqref{double} is mainly used for regularity purposes, but it also restricts the growth of $K$ in the following way: If $x\geq R$, then there exists a positive integer $N$ such that $R2^{N-1}\leq x\leq R2^N$. Thus
	$$
	K(x)\leq \alpha K(x/2)\leq \alpha^2K(x/2^2)\leq\ldots
	\leq \alpha^{N}K(x/2^{N})\leq \alpha^{N} K(R).
	$$
Using $(N-1)\log 2\leq \log\frac{x}{R}\leq N\log 2$, this gives rise to
	$$
	K(x)\leq O\left(x^\frac{\log\alpha}{\log 2}\right),
	\quad \alpha>1.
	$$
Due to our maximal growth rate, we have to make a technical restriction that $\alpha\leq 2$. The extremal value $\alpha=2$ works for the identity mapping. The power functions $x\mapsto x^a$, $a\in (0,1)$, satisfy the doubling condition for any $\alpha\geq 2^a$. The function $x\mapsto\log x$
works here as well. The non-permitted extremal value $\alpha=1$ works for constant functions only. 	
	
A countable collection of Euclidean discs $D(z_n,r_n)$ for which $z_n\in \mathbb{C},|z_n|\rightarrow \infty, r_n>0$, and
\begin{equation}\label{logsum}
	\sum_n \frac{r_n}{K(|z_n|)}<\infty,
	\end{equation}
is called a $K$-set. The case $K(x)=x$ corresponds to the $\E$-set, while the non-permitted case that $K$ is a constant function corresponds to the $R$-set. Note that \eqref{logsum} is not a restriction for the quantity of the points $z_n$.

For any $\phi\in[0,2\pi]$ and any $c>0$, we denote
	$$
	\Lambda(K,\phi,c)=\left\{z\in\C : |\Im\left(ze^{-i\phi}\right)|\leq cK
\left(\Re(ze^{-i\phi})\right),\, \Re\left(ze^{-i\phi}\right)\geq x_0\right\}.
	$$
In the case $\phi=0$, this domain reduces to
	$$
	\Lambda(K,0,c)=\left\{z\in\C : |y|\leq cK(x),\, x\geq x_0\right\},
	$$
and the boundary of it consists of two easily accessible curves of the form
	\begin{equation}\label{easy-curves}
	y=\pm cK(x),\quad x\geq x_0,
	\end{equation}
around the positive real axis.

We are now ready to state and prove the first of our main results, which contains the $\E$-set
as a special case. Indeed, if $K(x)=x$, then the choices $\phi\in\{0,\pi/2,\pi,3\pi/2\}$
will cover all rays emanating from the origin.

\begin{theorem}\label{thm1}
Let $U$ be a $K$-set, and let $\phi\in[0,2\pi]$.
Then the set $C\subset (0,\infty)$ of values $c$ for which the curve $\partial \Lambda(K,\phi,c)$
meets infinitely many discs $D(z_n,r_n)$ has measure zero.
Moreover, the projection $E$ of $U$ onto the interval $[1,\infty)$ satisfies $\int_E\frac{dx}{K(x)}<\infty$.
\end{theorem}

\begin{proof}
By appealing to a rotation, we may suppose that $\phi=0$. Then the boundary $\partial\Lambda(K,\phi,c)$ consists of the curves in \eqref{easy-curves}. It suffices to consider the curves
$y=cK(x)$ only, where $c\in I$ and $I\subset (0,\infty)$ is a compact set. Indeed, we may cover the
interval $(0,\infty)$ by countably many compact sets, and, by countable additivity of the Lebesgue measure,
a countable union of sets of measure zero has measure zero.

Without loss of generality, we may suppose that the points $z_n=x_n+iy_n$ are pairwise distinct and
organized by increasing modulus. If the sequence $\{x_n\}$ is bounded from above by a constant $M>0$, then
the points $z_n$ all lie in the half-plane $\Re z\leq M$. In this case each curve $y=cK(x)$ meets at
most finitely many discs $D_n=D(z_n,r_n)$ for all $c>0$. Thus we
may suppose that $\{x_n\}$ has a subsequence, denoted again by $\{x_n\}$, for which $x_n\to \infty$.

Let $\veps>0$. By \eqref{logsum}, we may choose $N(\veps)\in\N$ large enough such that
    \begin{equation}\label{epsilon-assumption}
    \sum_{n=N(\veps)}^\infty \frac{r_n}{K(|z_n|)}<\veps.
    \end{equation}
Since $x_n\to\infty$, we may further suppose that $x_n\geq R$ for all $n\geq N(\veps)$.
On the other hand, by once again appealing to a subsequence,
if necessary, we may suppose that the points $z_n$ are located in between the straight lines $y=\pm M x$, where $M>0$ depends only on the set $I$. For otherwise no curve of the form $y=c K(x)$
will certainly meet infinitely many discs $D_n$. Moreover, due to our
global growth restriction $K(x)=O(x)$, we may suppose that $K(x)\leq Mx$ for all $x\geq R$.

Using the monotonicity of $K$, we obtain
	$$
	K(|z_n|)=K\left(\sqrt{x_n^2+y_n^2}\right)\leq K(|x_n|+|y_n|)
	\leq K((1+M)x_n).
	$$
If $M\leq 1$, then	the doubling condition yields
	$$
	K(|z_n|)\leq K(2x_n)\leq \alpha K(x_n),
	$$
while if $M>1$, then there exists a positive integer $N$ such that
$2^{N-1}\leq 1+M\leq 2^N$, and so
	$$
	K(|z_n|)\leq \alpha K((1+M)x_n/2)\leq \ldots\leq \alpha^N K(x_n).
	$$
Thus, for all $M>0$, we find that
	\begin{equation}\label{zx}
	\frac{r_n}{K(|z_n|)}\geq \frac{r_n}{\alpha^N K(x_n)}\geq\frac{r_n}{\alpha^N Mx_n}.
	\end{equation}
Keeping \eqref{logsum} in mind, we see that the sequence $\big\{\frac{r_n}{x_n}\big\}$ tends to $0$.
Hence, by choosing a bigger $N(\veps)$ in \eqref{epsilon-assumption}, if necessary, we may suppose that
	\begin{equation}\label{xn-assumption}
	x_n-r_n\geq x_n/2\geq e,\quad n\geq N(\veps).
	\end{equation}

\begin{figure}[H]\label{concave-plane}
    \begin{center}
    \begin{tikzpicture}
    \draw[->](0,0)--(10,0)node[left,below]{$x$};
    \draw[->](1,-1)--(1,5)node[right]{$y$};
    \draw[domain=1:9] plot(\x,{ln(\x)})node[right,font=\tiny]{$y=c_1K(x)$};
     \draw[domain=1:9] plot(\x,{2*ln(\x)})node[right,font=\tiny]{$y=c_2K(x)$};
    \draw[thick,blue](7,2.9) circle [radius=26.5pt];
    \draw[thin](7,2.9) circle [radius=0.5pt]node[right]{$z_n$};
    \draw[thick](6.7,3.8) circle [radius=0.5pt]node[above=3pt,font=\tiny]{$(a_2,c_2K(a_2))$};
    \draw[thick](7.2,2) circle [radius=0.6pt]node[below=3pt,font=\tiny]{$(a_1,c_1K(a_1))$};
    \draw[-,dashed](1,0)--(7,2.9);
     \draw[-,dashed](7,2.9) to node[right,font=\tiny]{$r_n$} (6.7,3.8);
      \draw[-,dashed](7,2.9) to node[right,font=\tiny]{$r_n$} (7.2,2);
     \draw[-,dashed](1,0)--(6.7,3.8);
     \draw[-,dashed](1,0)--(7.2,2);
    \end{tikzpicture}
    \end{center}
	\begin{quote}
    \caption{Concave case in the complex plane.}
    \end{quote}
    \end{figure}
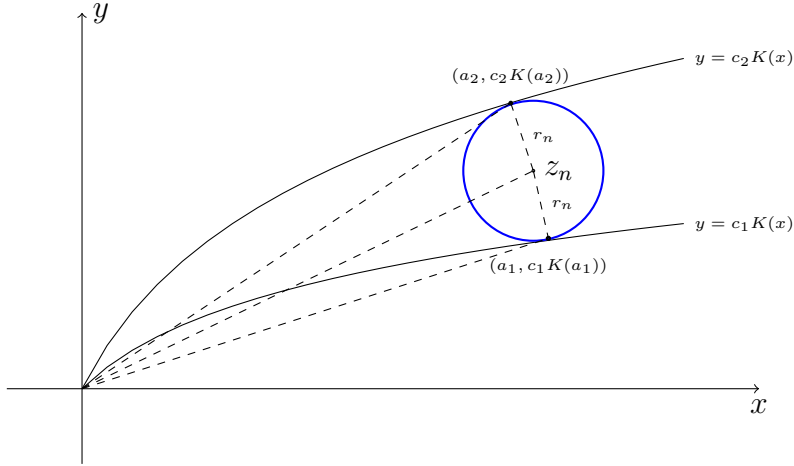

Suppose that a disc $D_n$, $n\geq N(\veps)$, lies asymptotically between the increasing curves
$y=c_1 K(x)$ and $y=c_2 K(x)$, where $c_2>c_1>0$.
Let $(a_i,c_i K(a_i))$, $i=1,2$, denote the intersection points. Since $K(x)$ is increasing, continuous and concave,
the intersection points satisfy
	\begin{equation}\label{intersection-points}
	x_n<a_1<x_n+r_n\quad\textnormal{and}\quad x_n-r_n<a_2<x_n.
	\end{equation}
This can be visualized geometrically by sketching an asymptotic square
	$$
	S_n=\{z\in\C : x_n-r_n\leq \Re (z)
	\leq x_n+r_n,\ y_n-r_n\leq \Im (z)\leq y_n+r_n\}
	$$
around the disc $D_n$. Then,
	\begin{eqnarray*}
	c_2K (x_n-r_n)&\leq& c_2K (a_2)\leq y_n+r_n,\\
	c_1 K (x_n+r_n)&\geq& c_1 K (a_1)\geq y_n-r_n,
	\end{eqnarray*}
which gives raise to
	$$
	c_2-c_1\leq \frac{y_n+r_n}{K (x_n-r_n)}-\frac{y_n-r_n}{K (x_n+r_n)}.
	$$
The curve $y=cK (x)$ which travels throught the center point $z_n$ of $D_n$ satisfies $c=y_n/K (x_n)$,
where $c_1<c<c_2$. Since $K(x)$ is increasing, continuous and concave, we see that
	$$
	y_n-r_n\leq cK (x_n-r_n)\leq cK( x_n+r_n)\leq y_n+r_n.
	$$
Using \eqref{xn-assumption}, this gives us
	\begin{equation}\label{cc}
	\begin{split}
	c_2-c_1&\leq \frac{y_n-r_n+2r_n}{K (x_n-r_n)}-\frac{y_n+r_n-2r_n}{K (x_n+r_n)}\\
	&\leq \frac{4r_n}{K (x_n-r_n)}\leq \frac{4r_n}{K (x_n/2)}\leq \frac{4\alpha r_n}{K (x_n)}.
	\end{split}
	\end{equation}

Let $\varepsilon>0$ and $N(\veps)\in\N$ be as in \eqref{epsilon-assumption}, and let $J_\varepsilon\subset I$ denote the set of values $c$ for which $y=cK(x)$ meets at least one disc $D_n$ with $n\geq N(\veps)$.
Then \eqref{epsilon-assumption}, \eqref{zx} and \eqref{cc} show that
	$$
	\int_{J_\varepsilon}dx\leq \sum_{n=N(\varepsilon)}\frac{4\alpha r_n}{K (x_n)}
	\leq \sum_{n=N(\varepsilon)}\frac{4\alpha^{N+1} r_n}{K (|z_n|)}<4\alpha^{N+1}\varepsilon.
	$$
Here $N$ depends only on $M$, which in turn depends only on the compact set~$I$. Since
$C\cap I$ is clearly contained in $J_\varepsilon$ for every $\varepsilon>0$, it follows that
$C\cap I$ has measure zero.

The remaining assertion follows from
	\begin{eqnarray*}
	\int_E\frac{dx}{K(x)} &\leq& \sum_n\int_{|z_n|-r_n}^{|z_n|+r_n}\frac{dx}{K (x)}+O(1)\leq
	\sum_n \frac{2r_n}{K (|z_n|-r_n)}+O(1)\\
	&\leq& \sum_n \frac{2r_n}{K (|z_n|/2)}+O(1)\leq \sum_n \frac{2\alpha r_n}{K(|z_n|)}+O(1)<\infty,
	\end{eqnarray*}
where we have used \eqref{xn-assumption}.
\end{proof}

\begin{remark}
The compact set $I\subset (0,\infty)$ in the previous proof is required essentially in the extremal case $K(x)=x$ only. If $K(x)/x\to 0$ as $x\to\infty$, then we may suppose that all points $z_n$ are in between the curves $y=\pm x$, in which case we may choose $M=1=N$.
\end{remark}

\subsection{Convex curves}\label{L}

Let $L:(0,\infty)\to (0,\infty)$ be a strictly decreasing, continuous, and convex function such that $L(x)\to 0^+$ as $x\to \infty$. Moreover, we suppose that $L$ satisfies the following doubling condition: There exist constants $\beta<1$ and $R\geq 1$ such that $L(2x) \geq \beta L(x)$ for all $x\geq R$. For example, the functions $x\mapsto x^{-p}$ satisfy all of the
above conditions for all $p>0$.

The doubling condition in question induces a restriction on how fast $L(x)$ can tend to zero. Indeed, by reversing the reasoning for the function $K$ in the beginning of Section~\ref{K}, we easily obtain
	$$
	L(x)\geq O\left(x^\frac{\log\beta}{\log 2}\right),
	\quad \beta<1.
	$$
So, the functions $x\mapsto x^{-p}$, $p>0$, are extremal in this sense.
	
A countable collection of Euclidean discs $D(z_n,r_n)$ for which $z_n\in \mathbb{C},|z_n|\rightarrow \infty, r_n>0$, and
\begin{equation}\label{logsum2}
	\sum_n \frac{r_n}{L(|z_n|)}<\infty,
	\end{equation}
is called an $L$-set. The non-permitted case that $L$ is a constant function corresponds to the $R$-set. In the extremal case $L(x)=x^{-p}$,
$p>0$, the radii $r_n$ must tend to zero very fast, which in turn leads to a very thin projection set on the interval $[1,\infty)$.

For any $\phi\in[0,2\pi]$ and any $c>0$, we denote
	$$
	\Lambda(L,\phi,c)=\left\{z\in\C : |\Im\left(ze^{-i\phi}\right)|\leq cL
\left(\Re(ze^{-i\phi})\right),\, \Re\left(ze^{-i\phi}\right)\geq R\right\}.
	$$
In the case $\phi=0$, this domain reduces to
	$$
	\Lambda(L,0,c)=\left\{z\in\C : |y|\leq cL(x),\, x\geq R\right\},
	$$
and the boundary of it consists of two easily accessible curves of the form
	\begin{equation}\label{easy-curves2}
	y=\pm cL(x),\quad x\geq R,
	\end{equation}
around the positive real axis.

\begin{theorem}\label{thm2}
Let $U$ be an $L$-set, and let $\phi\in[0,2\pi]$. Then the set $C\subset (0,\infty)$ of values $c$
for which the boundary $\partial\Lambda(L,\phi,c)$ meets infinitely many discs $D(z_n,r_n)$ has
measure zero. Moreover, the projection $E$ of $U$ onto the interval $[1,\infty)$ satisfies $\int_E\frac{dx}{L(x)}<\infty$.
\end{theorem}

\begin{proof}
We follow the same method used in proving Theorem~\ref{thm1}. We may suppose that $\phi=0$, in which case the boundary curves are of the form \eqref{easy-curves2}. If $z_n=x_n+iy_n$, we
may suppose that $\{x_n\}$ has a subsequence, denoted again by $\{x_n\}$, for which $x_n\to \infty$. Since $r_n\to 0$, the inequalities in \eqref{xn-assumption} require no further convincing in this case. Since $L(x)\to 0$ as $x\to\infty$, we may suppose that the points $z_n$ lie in between the lines $y=\pm x$. Therefore, we may directly consider different values of $c\in (0,\infty)$ without having to restrict to one compact subset of $(0,\infty)$ at the time.
Let $\veps>0$, and choose $N(\veps)\in\N$ large enough such that
    \begin{equation}\label{epsilon-assumption2}
    \sum_{n=N(\veps)}^\infty \frac{r_n}{L(|z_n|)}<\veps.
    \end{equation}
This is an analogue of \eqref{epsilon-assumption}.

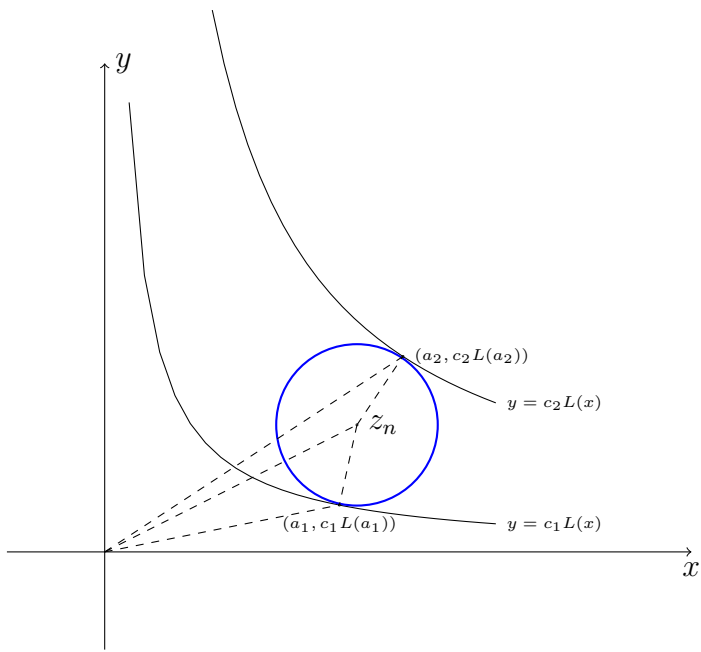
\begin{figure}[H]\label{convex-plane}
    \begin{center}
    \begin{tikzpicture}[scale=1.3]
    \draw[->](-1,0)--(6,0)node[left,below]{$x$};
    \draw[->](0,-1)--(0,5)node[right]{$y$};
    \draw[domain=1.1:4] plot(\x,{6.1/\x})node[right,font=\tiny]{$y=c_2L(x)$};
    \draw[domain=0.25:4] plot(\x,{1.15/\x})node[right,font=\tiny]{$y=c_1L(x)$};
    \draw[thick,blue](2.58,1.3) circle [radius=23.5pt];
    \draw[thick](2.58,1.3) circle [radius=0.1pt]node[right=0.001pt]{$z_n$};
    \draw[thick](2.4,0.48) circle [radius=0.1pt]node[below=0.01pt,font=\tiny]{$(a_1,c_1L(a_1))$};
    \draw[thick](3.05,2) circle [radius=0.1pt]node[above=0.1pt,font=\tiny,right]{$(a_2,c_2L(a_2))$};
    \draw[-,dashed](0,0)--(2.58,1.3);
    \draw[-,dashed](2.58,1.3)to node[above=0.02pt,font=\tiny]{} (2.4,0.48);
    \draw[-,dashed](2.58,1.3) to node[above,font=\tiny]{} (3.05,2);
   \draw[-,dashed](0,0)--(2.4,0.48);
    \draw[-,dashed](0,0)--(3.05,2);
    \end{tikzpicture}
    \end{center}
	\begin{quote}
    \caption{Convex case in the complex plane.}
    \end{quote}
    \end{figure}

Suppose that a disc $D_n$ lies asymptotically between the decreasing curves
$y=c_1 L(x)$ and $y=c_2 L(x)$, where $c_2>c_1>0$.
Let $(a_i,c_i L(a_i))$, $i=1,2$, denote the intersection points. Since $L(x)$ is decreasing, continuous and convex,
the intersection points satisfy
	\begin{equation*}
	x_n-r_n<a_1<x_n\quad\textnormal{and}\quad x_n<a_2<x_n+r_n.
	\end{equation*}
This can be visualized geometrically by sketching an asymptotic square
	$$
	S_n=\{z\in\C : x_n-r_n\leq \Re (z)\leq x_n+r_n,\ y_n-r_n\leq \Im (z)\leq y_n+r_n\}
	$$
around the disc $D_n$. Then,
	\begin{eqnarray*}
	c_2L (x_n+r_n)&\leq& c_2L (a_2)\leq y_n+r_n,\\
	c_1 L (x_n-r_n)&\geq& c_1 L (a_1)\geq y_n-r_n.
	\end{eqnarray*}
The analogue of \eqref{cc} is now
	\begin{equation}\label{cc2}
	\begin{split}
	c_2-c_1&\leq\frac{y_n-r_n+2r_n}{L(x_n+r_n)}-\frac{y_n+r_n-2r_n}{L(x_n-r_n)}\\
	&\leq \frac{4r_n}{L(x_n+r_n)}\leq \frac{4r_n}{L(2x_n)}\leq \frac{4r_n}{\beta L(x_n)}
	\leq \frac{4r_n}{\beta L(|z_n|)}.
	\end{split}
	\end{equation}

Let $C_\varepsilon\subset (0,\infty)$ denote the set of values $c$ for which the curve
$y=cL(x)$ meets at least one of the discs $D_n$, where $n\geq N(\varepsilon)$.
By \eqref{epsilon-assumption2} and \eqref{cc2},
	$$
	\int_{C_\varepsilon}dx\leq \sum_{n=N(\varepsilon)}^\infty \frac{4r_n}{\beta L(|z_n|)}
	<\frac{4}{\beta}\varepsilon.
	$$
Since $C$ is a subset of $C_\varepsilon$ for any $\varepsilon>0$, we conclude that $\int_Cdx=0$.

The remaining assertion follows from
	\begin{eqnarray*}
	\int_E\frac{dx}{L(x)} &\leq& \sum_n\int_{|z_n|-r_n}^{|z_n|+r_n}\frac{dx}{L(x)}+O(1)\leq
	\sum_n \frac{2r_n}{L(|z_n|+r_n)}+O(1)\\
	&\leq& \sum_n \frac{2r_n}{L(2|z_n|)}+O(1)\leq \sum_n \frac{2r_n}{\beta L(|z_n|)}+O(1)<\infty,
	\end{eqnarray*}
and the proof is complete.
\end{proof}

\subsection{The constant case}\label{C}

In order to avoid unnecessary complications in the geometric reasoning,
we have intentionally assumed strict monotonicity on $K(x)$ and on $L(x)$.
However, either proof goes through if these functions  are constant functions.
The curves in this situation are just half-lines, and the associated collection of
discs constitutes
an $R$-set. This gives rise to the following corollary.

\begin{corollary}\label{cor1}
Let $k\in\R$. Then the set $C\subset\R$ of values $c$ for which the line $y=kx+c$
or the line $x=c$ meets infinitely many
discs of a given $R$-set has measure zero.
\end{corollary}

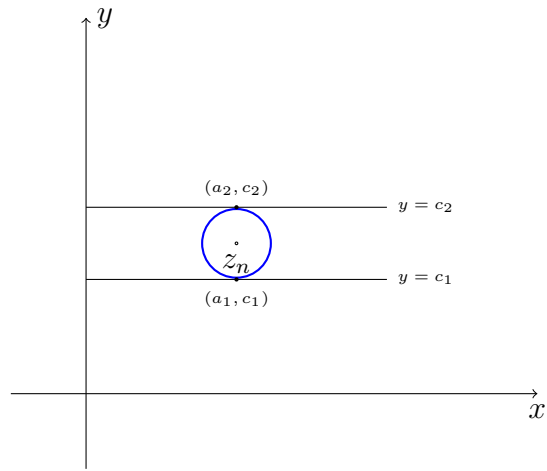
\begin{figure}[H]\label{measure}
    \begin{center}
    \begin{tikzpicture}[scale=1]
    \draw[->](-1,0)--(6,0)node[left,below]{$x$};
    \draw[->](0,-1)--(0,5)node[right]{$y$};
    \draw[domain=0:4] plot(\x,2.48)node[right,font=\tiny]{$y=c_2$};
    \draw[domain=0:4] plot(\x,1.52)node[right,font=\tiny]{$y=c_1$};
    \draw[thick,blue](2,2) circle [radius=13pt];
    \draw[thin](2,2) circle [radius=0.5pt]node[below]{$z_n$};
    \draw[thick](2,2.48) circle [radius=0.4pt]node[above,font=\tiny]{$(a_2,c_2)$};
    \draw[thick](2,1.52) circle [radius=0.4pt]node[below,font=\tiny]{$(a_1,c_1)$};
    \end{tikzpicture}
    \end{center}
	\begin{quote}
    \caption{Constant case in the complex plane.}
    \end{quote}
    \end{figure}

The following two examples illustrate the sharpness of Corollary~\ref{cor1} in regard to
the size of the set $C$.

\begin{example}
We construct an $R$-set $U$ by modifying the construction for the classical Cantor set $\Ce$.
For brevity, denote $I(x,r)=(x-r,x+r)$. Choose any two open discs in $\C$ whose vertical
projections on $\R$ are $I(1/6,2/6)$ and $I(5/6,2/6)$, and include them in the collection $U$. Next,
choose any four open discs in $\C$ whose vertical projections on $\R$ are
	$$
	I(1/18,2/18),\ I(5/18,2/18),\ I(13/18,2/18),\ I(17/18,2/18),
	$$
and include them in $U$. We continue in this way so that on the $k$th step we have $2^k$ open discs
whose vertical projections on $\R$ are open intervals of length $2/3^k$ and with the same center points as the closed intervals that are used in constructing the Cantor set $\Ce$. The collection $U$ contains all such discs
as $k$ tends to infinity, and hence their vertical projections form an open cover for $\Ce$.  The sum
of the diameters of the discs in $U$ is
	$$
	\sum_{k=1}^\infty 2^k\cdot\frac{2}{3^k}=\frac{4}{3}\sum_{k=0}^\infty \left(\frac23\right)^k=4,
	$$
and hence $U$ is an $R$-set. Now, if $c\in\Ce$, then the vertical line $x=c$ meets infinitely
many open discs in  $U$. Finally, it is well known that $\Ce$ is uncountable and has measure zero.

Note that this construction is independent on the density of the
center points $z_n$ of the discs in $U$ because the imaginary
parts of $z_n$'s can be chosen arbitrarily. Consequently, the exponent of convergence of the
sequence $\{z_n\}$ can be arbitrary or even infinite.
\end{example}

\begin{example}\label{example2}
Let $c_n=1/n$ for $n\in\N$. For each $n\in\N$ choose any sequence of discs $D(z_{n,k},r_{n,k})$ such that $\Im(z_{n,k})=c_n$, $|z_{n,k}|\to\infty$ as $k\to\infty$, $r_{n,k}>0$, and that
	$$
	\sum_{k=1}^\infty r_{n,k}\leq 2^{-n}.
	$$
Then each vertical line $y=c_n$ meets infinitely many discs, and the sum of the diameters of the discs in the entire collection $\{D(z_{n,k},r_{n,k})\}$ is $\leq 2$.
\end{example}

\subsection{Rapidly increasing convex curves}

Let $\L:(0,\infty)\to (0,\infty)$ be a strictly increasing, continuous
and convex function such that
	\begin{equation}\label{fast}
	\lim_{x\to\infty}\frac{x}{\L(x)}=0,
	\end{equation}
and let $R>0$ be such that $\L(x)\geq 1$ for all $x\geq R$.
In addition, we assume that $\L$ satisfies the following
doubling-type property: For every function $\delta:(0,\infty)\to (0,\infty)$ with
$\delta(x)\to 0$ as $x\to \infty$ there exists a constant $\alpha=\alpha(\delta)>0$ such that
	\begin{equation}\label{dumpling-cond0}
	\lim_{x\to \infty}\frac{\L(x)}{\L(x(1-\delta(x)))}=
	\lim_{x\to \infty}\frac{\L(x)}{\L(x(1+\delta(x)))}=\alpha.
	\end{equation}	
For example, the functions $x\mapsto x^p$, $p>1$, satisfy all of the above for $\alpha=1$.
From the projection sets point of view, a restriction	
	\begin{equation}\label{integral-infinity}
	\int_1^\infty\frac{dx}{\L(x)}=\infty
	\end{equation}
is necessary because it then makes sense to consider projection sets $E\subset[1,\infty)$ for which $\int_E\frac{dx}{\L(x)}<\infty$. For example, the function $x\mapsto x\log(1+x)$ satisfies all of the above, including
\eqref{integral-infinity}.

Even if \eqref{integral-infinity} is not
satisfied, it still makes sense to require
	\begin{equation}\label{logsum-L}
	\sum_{n=1}^\infty \frac{r_n}{\L(|z_n|)}<\infty
	\end{equation}		
regarding a collection of Euclidean discs $D_n=D(z_n,r_n)$. For example,
if $r_n=n$, $z_n=n\log (1+n)$ and $\L(x)=x^2$, then \eqref{logsum-L} is valid, but the angle in which each disc $D_n$ subtends at the origin
is asymptotic to $2/\log (1+n)\to 0$ as $n\to\infty$. Consequently,
the discs $D_n$ cover only a small portion of the plane $\C$
even though their projection set covers the whole interval $[1,\infty)$.
This is due to the fact that $z_{n+1}-r_{n+1}<z_n+r_n$ for all $n\in\N$.

Analogously as above, a countable collection of Euclidean discs $D(z_n,r_n)$
satisfying \eqref{logsum-L} with $z_n\in \mathbb{C}, |z_n|\rightarrow \infty$ and $r_n>0$
is called an $\L$-set.

For any $\phi\in[0,2\pi]$ and any $c>0$, we denote
	$$
	\Lambda(\L,\phi,c)=\left\{z\in\C : |\Im\left(ze^{-i\phi}\right)|\leq c\L
\left(\Re(ze^{-i\phi})\right),\, \Re\left(ze^{-i\phi}\right)\geq R\right\}.
	$$
In the case $\phi=0$, this domain reduces to
	$$
	\Lambda(\L,0,c)=\left\{z\in\C : |y|\leq c\L(x),\, x\geq R\right\},
	$$
and the boundary of it consists of two easily accessible curves of the form
	\begin{equation}\label{easy-curves2/5}
	y=\pm c\L(x),\quad x\geq R,
	\end{equation}
around the positive real axis.

\begin{theorem}\label{thm2/5}
Let $U$ be an $\L$-set, and let $\phi\in[0,2\pi]$. Suppose that
	\begin{equation}\label{technical0}
	\lim_{n\to\infty}\frac{r_n}{\Re\left(z_ne^{-i\phi}\right)}=0.
	\end{equation}
Then the set $C\subset (0,\infty)$ of values $c$
for which the boundary $\partial\Lambda(\L,\phi,c)$ meets infinitely many discs $D(z_n,r_n)$ has
no interior points. Moreover, the projection $E$ of $U$ onto the interval $[1,\infty)$ satisfies
$\int_E\frac{dx}{\L(x)}<\infty$.
\end{theorem}

\begin{remark}
(a) The assertion ''no interior points'' in Theorem~\ref{thm2/5} is weaker than the corresponding
assertion ''zero measure'' in Theorems~\ref{thm1} and \ref{thm2}. This is due to the rapid growth
of the function $\L(x)$ that causes the following phenomenon: In \eqref{cc} and \eqref{cc2} the
upper bound for $c_2-c_1$ constitutes of terms whose sum converges, whereas in proving
Theorem~\ref{thm2/5} the best upper bound for $c_2-c_1$ constitues of terms tending to zero (possibly
very slowly). However, since the set $C$ has no interior points, 
there is a dense set of curves $\partial\Lambda(\L,\phi,c)$
which meet at most finitely many discs $D(z_n,r_n)$.

(b) If $\phi=0$ and $z_n=x_n+iy_n$, then the technical assumption \eqref{technical0} reduces to
	\begin{equation}\label{technical}
	\lim_{n\to\infty}\frac{r_n}{x_n}=0.
	\end{equation}
A slightly weaker assumption $x_n-r_n>0$ would only assure that the discs
$D_n$ lie entirely in the right half-plane, same as the curves $y=\pm c\L(x)$. The angle in
which each disc $D(z_n,r_n)$ subtends at the origin is $2\theta_n$, where $\sin\theta_n =\frac{r_n}{|z_n|}$.
It follows from \eqref{technical} that $\theta_n\to 0$ as $n\to\infty$.
\end{remark}

We sketch the proof of Theorem~\ref{thm2/5} in the case $\phi=0$, that is, we consider
the curves in \eqref{easy-curves2/5}. Denote $z_n=x_n+iy_n$. If there exists an $M>0$ such that
$x_n\leq M$ for all $n\in\N$, then $r_n\to 0$ from \eqref{technical}, and so the discs $D_n$ are
in the half-plane $\{z\in\C:
\Re(z)\leq M+1\}$, with finitely many possible exceptions. Thus every $y=c\L(x)$ meets at most
finitely many discs $D_n$. Consequently, we may suppose that $x_n\to\infty$ as $n\to\infty$.

\begin{figure}[H]\label{convex-plane}
    \begin{center}
    \begin{tikzpicture}[scale=0.8]
    \draw[->](-1,0)--(6,0)node[left,below]{$x$};
    \draw[->](0,-1)--(0,9)node[right]{$y$};
    \draw[domain=0:2] plot(\x,{2*(\x)^2})node[right,font=\tiny]{$y=c_2\L(x)$};
    \draw[domain=0:4] plot(\x,{0.5*(\x)^2})node[right,font=\tiny]{$y=c_1\L(x)$};
    \draw[thick,blue](2.58,6) circle [radius=23.5pt];
    \draw[thin](2.58,6) circle [radius=0.4pt]node[above=0.001pt]{$z_n$};
    \draw[thick](1.75,6.1) circle [radius=0.5pt]node[below=0.01pt,font=\tiny,left]{$(a_2,c_2\L(a_2))$};
    \draw[thick](3.4,5.75) circle [radius=0.5pt]node[above=0.1pt,font=\tiny,right]{$(a_1,c_1\L(a_1))$};
    \draw[-,dashed](0,0)--(2.58,6);
    \draw[-,dashed](2.58,6)to node[above=0.02pt,font=\tiny]{} (1.75,6.1);
    \draw[-,dashed](2.58,6) to node[above,font=\tiny]{} (3.4,5.75);
   \draw[-,dashed](0,0)--(1.75*1.2,6.1*1.2);
    \draw[-,dashed](0,0)--(3.4*1.2,5.75*1.2);
    \end{tikzpicture}
    \end{center}
	\begin{quote}
    \caption{Rapidly increasing convex curves.}
    \end{quote}
    \end{figure}
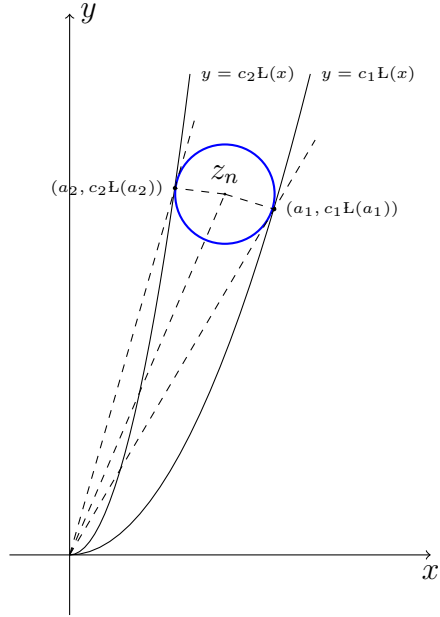

Suppose that a disc $D_n$ lies asymptotically between the curves $y=c_1 \L(x)$ and $y=c_2 \L(x)$, where $c_2>c_1>0$. Let $\zeta_j=a_j+ib_j$, $j=1,2$, be the corresponding points of intersection.
Let $c=y_n/\L(x_n)$. Then the curve $y=c\L(x)$ travels through the
center point $z_n$ of $D_n$, and $c_1<c<c_2$. Now
	\begin{eqnarray*}
	c_2-c_1 = &&\!\! \frac{b_2}{\L(a_2)}-\frac{b_1}{\L(a_1)}\leq
	\frac{y_n+r_n}{\L(x_n-r_n)}-\frac{y_n-r_n}{\L(x_n+r_n)}\\
	\leq &&\!\! \frac{2r_n}{\L(x_n-r_n)}
	+\frac{y_n}{\L(x_n)}
	\left(\frac{\L(x_n)}{\L(x_n-r_n)}-\frac{\L(x_n)}{\L(x_n+r_n)}\right)\\
	\leq &&\!\! \frac{2r_n}{x_n}\cdot\frac{x_n}{\L(x_n)}
	\cdot\frac{\L(x_n)}{\L\left(x_n\left(1-\frac{r_n}{x_n}\right)\right)}\\
	 &&\!\!+c\left(\frac{\L(x_n)}{\L\left(x_n\left(1-\frac{r_n}{x_n}\right)\right)}-\frac{\L(x_n)}{\L\left(x_n\left(1+\frac{r_n}{x_n}\right)\right)}\right),
	\end{eqnarray*}
where the upper bound tends to zero as $n\to\infty$ by \eqref{fast},
\eqref{dumpling-cond0} and \eqref{technical}.

Suppose on the contrary to the assertion that the set $C$ has an interior point $c_0$. Then
there exists a small constant $\varepsilon>0$ such that $J_\varepsilon:=[c_0-\varepsilon,c_0+\varepsilon]\subset C$.
Thus each curve $y=c\L(x)$ with $c\in J_\varepsilon$ meets infinitely many
discs $D_n$. Since the interval $J_\varepsilon$
is uncountable, while the sequence $\{D_n\}$ is only countable, there must exist
$d_1,d_2\in J_\varepsilon$ with $d_2-d_1=\delta>0$ and a subsequence $\{D_{n_k}\}$
such that both of $y=d_j\L(x)$, $j=1,2$, meet every disc $D_{n_k}$. In fact, using convexity,
it follows that each curve $y=d\L(x)$ with $d\in [d_1,d_2]$ meets every disc $D_{n_k}$.
From the discussion above, we may find an $N(\delta)\in\N$ such that
	$$
	0<\delta=d_2-d_1\leq c_2-c_1<\delta/2,\quad n_k\geq N(\delta),
	$$
which is a contradiction. Hence $C$ has no interior points.

The remaining assertion $\int_E\frac{dx}{\L(x)}<\infty$ follows 
easily from \eqref{dumpling-cond0}, \eqref{logsum-L} and \eqref{technical}.


\section{Results for curves in the unit disc}\label{disc-curves}

The continuous transition method presented in Section~\ref{plane-curves}
can be formulated in the unit disc $\D$ as well. Just as above, the theory divides into concave and convex situations. We begin with some preliminary remarks.

\subsection{Preliminary remarks}\label{remarks}

A typical size for an exceptional set $E\subset [0,1)$ in
value distribution theory is finite logarithmic measure in the sense of $\int_{E}\frac{dx}{1-x}<\infty$. In particular, if a sequence of Euclidean
discs $\Delta(z_n,r_n)\subset\D$ satisfies
	\begin{equation}\label{sum-D}
	\sum_n\frac{r_n}{1-|z_n|}<\infty,
	\end{equation}
then $\frac{r_n}{1-|z_n|}\to 0$ as $n\to\infty$, and the projection set $E\subset [0,1)$ satisfies
	\begin{eqnarray*}
	\int_{E}\frac{dx}{1-x}&\leq& \sum_n\int_{|z_n|-r_n}^{|z_n|+r_n}\frac{dx}{1-x}+O(1)\\
	&\leq & \sum_n \frac{2r_n}{(1-|z_n|)\left(1-\frac{r_n}{1-|z_n|}\right)}+O(1)<\infty.
	\end{eqnarray*}
Thus the condition \eqref{sum-D} is also natural in the value distribution theory.

Recall that a Stolz angle with vertex at $\zeta\in\partial\D$ is a set
	$$
	S(\zeta,c)=\left\{z\in\D : |1-\overline{\zeta}z|<c(1-|z|)\right\},
	$$
where $c>1$ is some constant \cite{Girela}. For $c=1$ this set reduces to the line segment $[0,1)$, while for $c<1$ the set is empty.
If the points $z_n$ are in a Stolz angle with vertex at $z=1$, then the quantities $1-|z_n|$ and $|1-z_n|$ are uniformly comparable \cite[Lemma~4]{Girela0}. The condition \eqref{sum-D} then implies that the	 subtending angles at the point $z=1$ have a finite sum. Hence, for any $\veps>0$ there exists a positive integer $N(\veps)$ such that
	$$
	\sum_{n\geq N(\veps)}\frac{r_n}{|1-z_n|}<\veps.
	$$
This implies that the set of points $c\in\R$ for which the lines $y=c(x-1)$
(through the point $z=1$) meets infinitely many discs $\Delta(z_n,r_n)$ has linear measure zero. From this point of view, such a collection of discs $\Delta(z_n,r_n)$ is a unit disc analogue for the $\E$-set.
For more details, see Section~\ref{Stolz}.

Analogous to $R$-sets, typical situations when the diameters of a given collection of discs in $\D$ have a finite sum arise, for example, when $\{z_n\}$ is a Blaschke sequence ($r_n=1-|z_n|$) or when $\{z_n\}$ has exponent of convergence
$\lambda\geq 0$ ($r_n=(1-|z_n|)^{1+\lambda+\veps}$). These cases are very well-known.

As for the curves tending to a boundary point $\zeta\in\partial\D$, the boundaries of  Stolz angles
are a good starting point. In addition to Stolz angles, we take as our basic model the domains
	$$
	R(\zeta,\gamma,c)=\left\{z\in\D : |1-\overline{\zeta}z|<c(1-|z|)^\gamma\right\},
	$$
where $\gamma>0$ and $c>0$ are some constants \cite{Ahern,Girela}. The special case $\gamma=1/2$ corresponds to a horodisc at $\zeta$ \cite{Girela}. In general, $R(\zeta,\gamma,c)$ is a tangential domain for $\gamma\in (0,1)$ and a represents a zero angle at $\zeta$ for $\gamma>1$. For example, the angular domain between the curves $y=\pm (1-x)^p$, $p>1$, is a zero angle at $z=1$ because both curves have horizontal tangents at the point $(1,0)$. If $\zeta=1$, the boundary $\partial R(1,\gamma,c)$ in the upper half-disc is concave for $0<\gamma<1$ and convex for $\gamma>1$. The extremal case $\gamma=1$ in between these two cases corresponds to the constant case in Section~\ref{C}.

\subsection{Concave curves}\label{k}

Let $l:(0,1)\to (0,1)$ be a strictly increasing, continuous and concave  function such that
	\begin{equation}\label{k-zero}
	\lim_{x\to 1^-}\frac{1-x}{l(1-x)}=0.
	\end{equation}
In addition, we assume that $l$ satisfies the following
doubling-type property: For every function $\delta:(0,1)\to (0,1)$ with
$\delta(x)\to 0$ as $x\to 1^-$ and for any $\gamma>0$ there exists a constant $\tau=\tau(\delta,\gamma)>0$ such that
	\begin{equation}\label{dumpling-cond}
	\lim_{x\to 1^-}\frac{l(1-x)}{l((1-x)(\gamma-\delta(x)))}=
	\lim_{x\to 1^-}\frac{l(1-x)}{l((1-x)(\gamma+\delta(x)))}=\tau.
	\end{equation}	
For example, the functions $x\mapsto x^a$, $0<a<1$, satisfy all of the conditions above. Since the functions $x\mapsto x^a$, $a\geq 1$, also satisfy
\eqref{dumpling-cond}, we deduce that \eqref{dumpling-cond} can be
satisfied by both convex and concave functions. The function $x\mapsto \exp(-1/x)$ shows that not all continuous functions satisfy \eqref{dumpling-cond}, which can easily
be verified by choosing $\delta(x)=\sqrt{1-x}$.

A countable collection of Euclidean discs $\Delta(z_n,r_n)\subset\D$ for which
\begin{equation}\label{logsum3}
	\sum_n \frac{r_n}{l(1-|z_n|)}<\infty,
	\end{equation}
is called an $l$-set. For any $\zeta\in\partial\D$ and any $c>0$, we denote
	$$
	\Gamma(l,\zeta,c)=\left\{z : |1-\overline{\zeta}z|< cl(1-|z|)\right\}.
	$$

\begin{theorem}\label{thm3}
Let $U$ be an $l$-set, and let $\zeta\in\partial\D$. Suppose that
	\begin{equation}\label{extra-assumption}
	\lim_{n\to\infty}\frac{r_n}{1-|z_n|}=0.
	\end{equation}
Then the set $C\subset (0,\infty)$ of values $c$ for which the curve $\partial\Gamma(l,\zeta,c)$
meets infinitely many discs $\Delta(z_n,r_n)$ has no interior points. Moreover, the projection
$E$ of $U$ onto the interval $[0,1)$ satisfies $\int_E\frac{dx}{l(1-x)}<\infty$.
\end{theorem}

\begin{remark}
(a) The condition \eqref{extra-assumption} resembles the condition
\eqref{technical}. The conclusion ''no interior points'' is the same as the
conclusion of Theorem~\ref{thm2/5}. This is again a result of the fact that
the terms constituting the upper bound for $c_2-c_1$ tend to zero, but possibly very slowly.

(b) The angle in which each disc $\Delta(z_n,r_n)$ subtends at the
boundary point $e^{i\arg(z_n)}$ is $2\varphi_n$, where $\sin\varphi_n =\frac{r_n}{1-|z_n|}$.
Therefore, it follows from \eqref{extra-assumption} that $\varphi_n\to 0$ as $n\to\infty$.

(c) The condition $\int_E\frac{dx}{l(1-x)}<\infty$ makes sense only
if the function $l$ satisfies $\int_0^1\frac{dx}{l(1-x)}=\infty$. This happens, for example,
when $l(x)=x\log\frac{e}{x}$. However, this
property for $l$ is not essential because the crux
of the result lies in representing the size of $U$ in the form \eqref{logsum3}.
\end{remark}

\begin{example}
The points $z_n=(1+e^{i/n})/2$ are on the horocycle $|z-1/2|=1/2$. This
corresponds to the situation when $l(x)=\sqrt{x}$. We have
	$$
	1-|z_n|^2=(1-\cos(1/n))/2.
	$$
Choosing $r_n=1/n^{2}\log^2(1+n)$, it follows that	
	$$
	\sum_{n=1}^\infty \frac{r_n}{\sqrt{1-|z_n|}}\leq
	\sum_{n=1}^\infty \frac{\sqrt{2}r_n}{\sqrt{1-|z_n|^2}}
	\asymp \sum_{n=1}^\infty \frac{1}{n\log^2(1+n)}<\infty.
	$$
Meanwhile,
	$$
	\frac{r_n}{1-|z_n|}\sim \frac{2r_n}{1-|z_n|^2}\sim
	8n^2r_n=\frac{8}{\log^2(1+n)}\to 0.
	$$
Thus the discs $\Delta(z_n,r_n)$ satisfy the assumptions
in Theorem~\ref{thm3}.
\end{example}

We sketch the proof of Theorem~\ref{thm3} in the case $\zeta=1$.
It suffices to consider the portion of $\partial\Gamma(l,0,c)$ in the upper half-disc,
call it $\Gamma^+(c)$ for short. Suppose that the disc $\Delta(z_n,r_n)$ lies asymptotically
between the curves $\Gamma^+(c_1)$ and $\Gamma^+(c_2)$, where $c_2>c_1$ and $z_n=x_n+iy_n$.
Denote the points of intersection on $\partial \Delta(z_n,r_n)$ by $\zeta_1$ and $\zeta_2$, respectively.

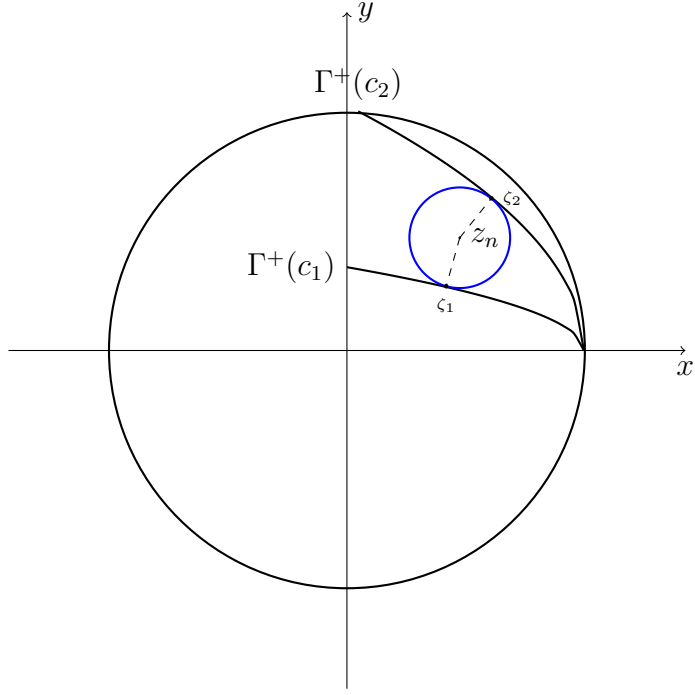
\begin{figure}[H]\label{measure}
    \begin{center}
    \begin{tikzpicture}[scale=1.5]
    \draw[->](-3,0)--(3,0)node[left,below]{$x$};
    \draw[->](0,-3)--(0,3)node[right]{$y$};
    \draw[domain=2.1:0.1,smooth,thick] plot(\x,{1.5*sqrt(2.1-\x)})node[left,above,thick]{$\Gamma^+(c_2)$};
    \draw[domain=2.1:0,smooth,thick] plot(\x,{0.51*sqrt(2.1-\x)})node[left,thick]{$\Gamma^+(c_1)$};
    \draw[thick,blue](1,1) circle [radius=12.7pt];
    \draw[thick](0,0) circle [radius=60pt];
    \draw[thin](1,1) circle [radius=0.2pt]node[right]{$z_n$};
   \draw[thick](1.28,1.35) circle [radius=0.3pt]node[right,thick,font=\tiny]{$\zeta_2$};
   \draw[thick](0.88,0.57) circle [radius=0.3pt]node[below,thick,font=\tiny]{$\zeta_1$};
   \draw[-,dashed](1,1)--(1.28,1.35);
   \draw[-,dashed](1,1)--(0.88,0.57);
    \end{tikzpicture}
    \end{center}
	\begin{quote}
    \caption{Concave case in the unit disc.}
    \end{quote}
    \end{figure}

From \eqref{k-zero}, we see that the curves $\Gamma^+(c_1)$ and $\Gamma^+(c_2)$ are essentially above the descending line $y=1-x$. Therefore, by
passing through a subsequence, if necessary, we may suppose that $y_n\geq 1-x_n$. Then, using concavity,
	\begin{eqnarray*}
	c_2-c_1=&&\!\!\frac{|1-\zeta_2|}{l(1-|\zeta_2|)}-\frac{|1-\zeta_1|}{l(1-|\zeta_1|)}\leq \frac{|1-z_n|+r_n}{l(1-|\zeta_2|)}-\frac{|1-z_n|-r_n}{l(1-|\zeta_1|)}\\
	\leq &&\!\! \frac{2r_n}{l(1-|\zeta_2|)}+|1-z_n|\left(\frac{1}{l(1-|\zeta_2|)}-\frac{1}{l(1-|\zeta_1|)}\right)\\
	\leq && \!\!\frac{2r_n}{l(1-|z_n|-r_n)}+|1-z_n|\left(\frac{1}{l(1-|z_n|-r_n)}-\frac{1}{l(1-|z_n|+r_n)}\right)\\
	=&& \!\!\frac{2r_n}{l(1-|z_n|)}\cdot\frac{l(1-|z_n|)}{l\left((1-|z_n|)\left(1-\frac{r_n}{1-|z_n|}\right)\right)}\\
	 &&\!\!+\frac{|1-z_n|}{l(1-|z_n|)}\left(\frac{l(1-|z_n|)}{l\left((1-|z_n|)\left(1-\frac{r_n}{1-|z_n|}\right)\right)}-\frac{l(1-|z_n|)}{l\left((1-|z_n|)\left(1+\frac{r_n}{1-|z_n|}\right)\right)}\right).
	\end{eqnarray*}
Here $|1-z_n|/l(1-|z_n|)=c$ for some $c\in (c_1,c_2)$, so that the
upper bound tends to zero as $n\to\infty$ by the assumptions
\eqref{dumpling-cond}, \eqref{logsum3} and \eqref{extra-assumption}.
The assertion that the set $C$ has no interior points follows now analogously as in the proof of
Theorem~\ref{thm2/5}.

The remaining assertion follows from
	\begin{eqnarray*}
	\int_{E}\frac{dx}{l(1-x)}&\leq& \sum_n\int_{|z_n|-r_n}^{|z_n|+r_n}\frac{dx}{l(1-x)}+O(1)\leq\sum_n \frac{2r_n}{l(1-|z_n|-r_n)}+O(1)\\
	&= & \sum_n \frac{2r_n}{l\left((1-|z_n|)\left(1-\frac{r_n}{1-|z_n|}\right)\right)}+O(1)<\infty,
	\end{eqnarray*}
where we have used \eqref{dumpling-cond}, \eqref{logsum3} and \eqref{extra-assumption}.

\subsection{Convex curves}\label{l}

Let $k:(0,1)\to (0,1)$ be a strictly increasing, continuous and convex function such that
	\begin{equation}\label{l-zero}
	\lim_{x\to 1^-}\frac{1-x}{k(1-x)}=\infty.
	\end{equation}
In addition, we assume that $k$ satisfies the
doubling-type property in \eqref{dumpling-cond}.
For example, the functions $x\mapsto x^a$, $a>1$, satisfy all of the conditions above.

A countable collection of Euclidean discs $\Delta(z_n,r_n)\subset\D$ for which
\begin{equation}\label{logsum4}
	\sum_n \frac{r_n}{k(1-|z_n|)}<\infty,
	\end{equation}
is called a $k$-set. For any $\zeta\in\partial\D$ and any $c>0$, we denote
	$$
	\Gamma(k,\zeta,c)=\left\{z : |1-\overline{\zeta}z|< ck(1-|z|)\right\}.
	$$

\begin{theorem}\label{thm4}
Let $U$ be a $k$-set, and let $\zeta\in\partial\D$. Then the set $C\subset (0,\infty)$
of values $c$ for which the curve $\partial\Gamma(k,\zeta,c)$
meets infinitely many discs $\Delta(z_n,r_n)$ has no interior points. Moreover, the projection $E$ of $U$
onto the interval $[0,1)$ satisfies $\int_E\frac{dx}{k(1-x)}<\infty$.
\end{theorem}

The proof is very similar to that of Theorem~\ref{thm3}, and hence
is omitted. The only
difference basically is that this time we do not need to assume
\eqref{extra-assumption} as it follows trivially from \eqref{logsum4}
by means of \eqref{l-zero}.

\begin{figure}[H]\label{measure}
    \begin{center}
    \begin{tikzpicture}[scale=1.5]
    \draw[->](-3,0)--(3,0)node[left,below]{$x$};
    \draw[->](0,-3)--(0,3)node[right]{$y$};
    \draw[domain=2.1:0.9,smooth,thick] plot(\x,{1.95*(2.1-\x)^2})node[left,above,thick]{$\Gamma^+(c_2)$};
    \draw[domain=2.1:0,smooth,thick] plot(\x,{0.29*(2.1-\x)^2})node[left,thick]{$\Gamma^+(c_1)$};
    \draw[thick,blue](0.9,1) circle [radius=12.7pt];
    \draw[thick](0,0) circle [radius=60pt];
    \draw[thin](0.9,1) circle [radius=0.18pt]node[right,above]{$z_n$};
   \draw[thick](1.33,1.15) circle [radius=0.3pt]node[right,thick,font=\tiny]{$\zeta_2$};
   \draw[thick](0.6,0.65) circle [radius=0.3pt]node[below,thick,font=\tiny]{$\zeta_1$};
   \draw[-,dashed](0.9,1)--(1.33,1.15);
   \draw[-,dashed](0.9,1)--(0.6,0.65);
    \end{tikzpicture}
    \end{center}
	\begin{quote}
    \caption{Convex case in the unit disc.}
    \end{quote}
    \end{figure}
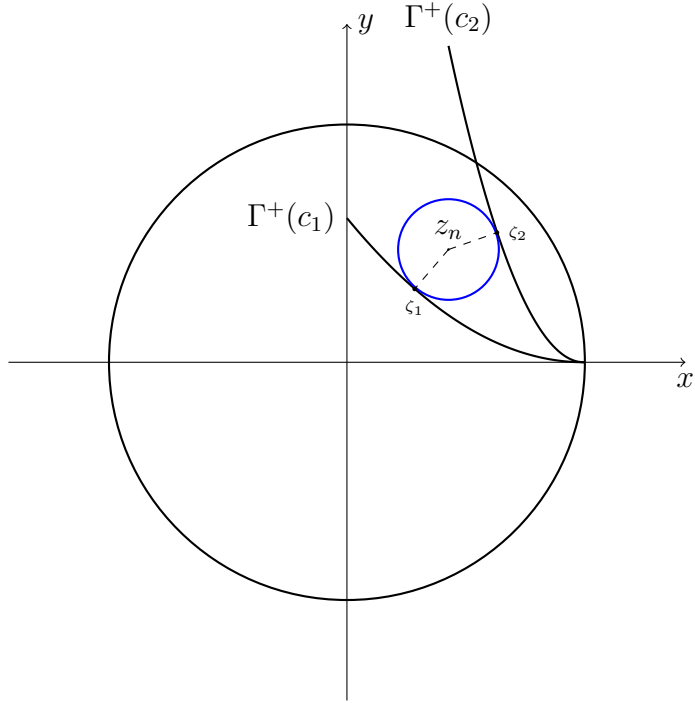
    
Finally we remark that, due to the condition \eqref{l-zero}, the
projection set $E\subset[0,1)$ satisfying $\int_E\frac{dx}{k(1-x)}<\infty$
has to be very thin.

\subsection{Stolz angles}\label{Stolz}

Here we consider the particular choice $k(x)=x^\gamma$, where $x\in (0,1)$ and $\gamma\geq 1$.
As discussed in Section~\ref{remarks}, the cases $\gamma=1$ and $\gamma>1$ correspond
to a Stolz angle and its subset, a zero angle, respectively.

\begin{theorem}
Let $\gamma\geq 1$ and $\zeta\in\partial\D$, and suppose that $U$ is a collection of Euclidean
discs $\Delta(z_n,r_n)\subset\D$ such that
	\begin{equation}\label{convergence-Stolz}
	\sum_{n=1}^\infty \frac{r_n}{(1-|z_n|)^\gamma}<\infty.
	\end{equation}
Then the set $C\subset (1,\infty)$ of values $c$ for which
$\partial\left\{z : |1-\overline{\zeta}z|< c(1-|z|)^\gamma\right\}$ meets infinitely many discs $\Delta(z_n,r_n)$ has measure zero. Moreover, the projection $E$ of $U$
onto the interval $[0,1)$ satisfies $\int_E\frac{dx}{(1-x)^\gamma}<\infty$.
\end{theorem}

We sketch the proof in the case $\zeta=1$ as follows. Since $|1-z_n|$ and $1-|z_n|$ are uniformly
comparable for $z_n$'s in a fixed Stolz angle (this is where we need $\gamma\geq 1$),
there exists a $K>0$ such that
	\begin{equation*}
	\begin{split}
	|1-z_n|&\left(\frac{1}{(1-|z_n|-r_n)^\gamma}-\frac{1}{(1-|z_n|+r_n)^\gamma}\right)\\
	& =|1-z_n|\gamma\int_{1-|z_n|-r_n}^{1-|z_n|+r_n}\frac{dx}{x^{1+\gamma}}
	\leq \frac{Kr_n}{(1-|z_n|-r_n)^\gamma}\\
	& = \frac{Kr_n}{(1-|z_n|)^\gamma\left(1-\frac{r_n}{1-|z_n|}\right)^\gamma}	
	\leq \frac{2Kr_n}{(1-|z_n|)^\gamma}
	\end{split}
	\end{equation*}
holds for all $n$ large enough. Following the proof of Theorem~\ref{thm3}, we have
	$$
	c_2-c_1\leq \frac{(4+2K)r_n}{(1-|z_n|)^\gamma}
	$$
for all $n$ large enough. Differing from the proof of Theorem~\ref{thm3}, the sum of the
terms in the upper bound converges by \eqref{convergence-Stolz}. To prove that the set $C$
is of measure zero, it suffices to follow the proof of Theorem~\ref{thm1} or the proof of
Theorem~\ref{thm2}. The remaining assertion is trivial.


\section{Applications}\label{ld-ld}

As applications of Sections~\ref{plane-curves}~and~\ref{disc-curves}, we discuss pointwise estimates
for logarithmic derivatives, logarithmic differences, and for
logarithmic $q$-differences. In addition, 
we discuss ways to avoid exceptional sets.

\subsection{Logarithmic derivatives in $\C$}

The paper \cite{Gundersen} is frequently cited in
the theory of complex differential equations. 
The extremal cases $K(x)\equiv x$ and $K(x)\equiv 1$ in Corollary~\ref{k-thm} below correspond to
Theorems~3 and 4 in \cite{Gundersen}, respectively. In fact, \cite[Theorem~4]{Gundersen}
is slightly improved because $r^\varepsilon$ in the upper bound is replaced with $\log^\alpha r$. We remind the reader of Theorem~\ref{thm1} according
to which a $K$-set can be avoided by almost every $K$-curve, and the
projection $E$ of a $K$-set on the interval $[1,\infty)$ satisfies
$\int_E\frac{dr}{K(r)}<\infty$.

\begin{corollary} \label{k-thm}
Let $f$ be a meromorphic function in $\C$, let $k$ and $j$ be integers such that $k>j\geq 0$,
and let $\alpha >1$ be a given constant. Suppose that $f^{(j)}\not\equiv 0$
and that
$K:[0,\infty)\to [0,\infty)$ is either the constant function $K(x)\equiv 1$ or a strictly increasing, continuous, concave function satisfying the doubling condition \eqref{double} and $K(x_0)=0$ for some $x_0\geq 0$.

Then there exists a $K$-set $U$ of discs $D(z_n,r_n)$ satisfying \eqref{logsum} and a constant $C>0$ depending only on $\alpha, k, j$ such that for all $z\not\in U$,
we have the following estimate (where $r=|z|$):
    \begin{equation}\label{k-thm-assertion}
    \left|\frac{f^{(k)}(z)}{f^{(j)}(z)}\right|
    \leq C\left(\frac{T(\alpha r,f)}{r}+n_j(\alpha r) \log^+n_j(\alpha r)\frac{\log^\alpha r}{K(r)}\right)^{k-j}.
    \end{equation}
Here $n_j(t)$ is the number of zeros and poles of $f^{(j)}$ in the disc $\{\zeta:|\zeta|\leq t\}$, counting multiplicities.
    \end{corollary}

We note that \cite[Theorem~1]{Chyzhykov0} is in the spirit of Corollary~\ref{k-thm} but
the estimate in it is obtained by other means and is valid outside of a fixed exceptional set that corresponds to the case $K(x)=x$. Meanwhile, the estimate in \eqref{k-thm-assertion} depends on the
size of the exceptional set determined by~$K$.

\bigskip
\noindent
\emph{Proof of Corollary~\ref{k-thm}.}
Let $\{a_m\}$ denote the sequence of zeros and poles of $f^{(j)}$ listed according to
multiplicity and ordered by increasing modulus. We follow the proof of 
\cite[Theorem~3]{Gundersen} up to (7.7), and set
    $$
    \mu_\nu=n(\alpha^{\nu+2})\quad\textnormal{and}\quad
    d_\nu=\frac{K(\alpha^\nu)}{(\log (\alpha^\nu))^{\alpha}},
    $$
where $\nu\geq \nu_0$ is an integer. We make this choice for the constants $d_\nu$
independently on whether $K(x)\equiv 1$ or not. We also suppose that $z$ is confined to the annulus 
$\A_\nu=\{\zeta:\alpha^\nu\leq |\zeta|< \alpha^{\nu+1}\}$, which
is equivalent to saying that
	$$
	\nu\log\alpha\leq \log r<(\nu+1)\log\alpha,\quad |z|=r.
	$$

Since $\alpha>1$ is fixed, there exists an integer $l\geq 1$ such that $\frac{1}{2^l}\leq \frac{1}{\alpha}<\frac{1}{2^{l-1}}$. We now make a standard use of Cartan's lemma as in \cite{Gundersen}. Indeed, if $z$ lies outside of Cartan's discs, then the following analogue of (7.9) in \cite{Gundersen} holds:
    \begin{eqnarray*} \label{ineq 1}
    \sum_{|a_m|\leq \alpha r}\frac{1}{|z-a_m|}
    &\leq& \sum_{m=1}^{\mu_\nu}\frac{1}{|z-b_m|}<\frac{\mu_\nu(\log \alpha^\nu)^{\alpha}}{K(\alpha^\nu)}\sum_{m=1}^{\mu_\nu}\frac{1}{m}\\
    &=& \frac{\mu_\nu(\log \alpha^\nu)^{\alpha}}{K(\frac{1}{\alpha}\alpha^{\nu+1})}\sum_{m=1}^{\mu_\nu}\frac{1}{m}\leq \mu_\nu(1+\log \mu_\nu)\frac{\log^{\alpha} r} {K(\frac{1}{2^l}\alpha^{\nu+1})}\\
    &\leq& \alpha^l\mu_\nu(1+\log \mu_\nu)\frac{\log^\alpha r}{K(r)} \leq \alpha^l n_j(\alpha^2 r)(1+\log n_j(\alpha^2 r))\frac{\log^\alpha r}{K(r)}\\
    &\leq& \alpha^{l+1} n_j(\alpha^2 r) \log n_j(\alpha^2 r)\frac{\log^\alpha r}{K(r)}.
    \end{eqnarray*}
The assertion \eqref{k-thm-assertion} follows from this similarly as in \cite[(7.10)]{Gundersen}.
Thus \eqref{k-thm-assertion} holds for a fixed $\nu\geq \nu_0$ when $z$ lies outside of Cartan's discs.

Similarly as in \cite{Gundersen}, we now consider all $\nu\geq\nu_0$,
where $\nu_0\geq 1$ is large enough.
Let $A_{\nu,1},\ldots,A_{\nu,l_\nu}$ denote precisely those Cartan
discs that intersect the annulus $\A_\nu.$
We mention that for some $\nu$ there might be no discs of this type.
It follows from Cartan's lemma that for each $\nu$ the total sum
of the diameters of the discs $A_{\nu,1},\ldots,A_{\nu,l_\nu}$
cannot exceed $4d_\nu$. Since $K(x)=O(x)$, we may choose $\nu_0$ large enough so that the
$\alpha^\nu>4d_\nu$ for all $\nu\geq\nu_0$. Then for $\nu\geq\nu_0$
the origin lies outside of the discs $A_{\nu,1},\ldots,A_{\nu,l_\nu}$.

Now let $D_n=D(z_n,r_n)$ denote the sequence $\{D_n\}$ of all the
discs $A_{\nu,i}$, where $\nu\geq\nu_0$ and $1\leq i\leq l_\nu$.
We have proved that \eqref{k-thm-assertion} holds if $z$ lies
outside of the discs $D_n$ and $\{\zeta:|\zeta|\leq \alpha^{\nu_0}\}$.
Moreover,
	\begin{eqnarray*}
	\sum_{n=1}^\infty\frac{r_n}{K(|z_n|)}
	&=& \sum_{\nu=\nu_0}^\infty \sum_{D_n\cap\A_\nu\neq\emptyset}\frac{r_n}{K(|z_n|)}\leq\sum_{\nu=\nu_0}^\infty\frac{2d_\nu}{K(\alpha^\nu-2d_\nu)}\\
	&=&\sum_{\nu=\nu_0}^\infty\frac{2d_\nu}{K(\alpha^\nu(1-2d_\nu/\alpha_\nu))}
	\leq \sum_{\nu=\nu_0}^\infty\frac{2d_\nu}{K(\alpha^\nu/2)}\\
	&\leq& \sum_{\nu=\nu_0}^\infty\frac{2\alpha d_\nu}{K(\alpha^\nu)}
	=\sum_{\nu=\nu_0}^\infty\frac{2\alpha}{\nu^\alpha (\log\alpha)^\alpha}
	<\infty.
	\end{eqnarray*}
This completes the proof. 
\hfill$\Box$

\subsection{Logarithmic derivatives in $\D$}

The extremal case $k(x)\equiv x$ in Corollary~\ref{k-thm-D} below corresponds to \cite[Theorem~3.1]{Chyzhykov}. We remind the reader of Theorem~\ref{thm4} according
to which a $k$-set can be avoided by almost every $k$-curve, and the
projection $E$ of a $k$-set on the interval $[0,1)$ satisfies
$\int_E\frac{dr}{k(1-r)}<\infty$.

\begin{corollary} \label{k-thm-D}
Let $f$ be a meromorphic function in $\D$, let $k$ and $j$ be integers such that $k>j\geq 0$,
let $\alpha\in (1,\infty)$ and $b\in (0,1)$ be given constants, and denote $s(r)=1-b(1-r)$. Suppose that $f^{(j)}\not\equiv 0$ and that
$k:(0,1)\to (0,1)$ is either the identity mapping $k(x)=x$ or a strictly increasing, continuous, convex function satisfying \eqref{l-zero} and the doubling-type condition \eqref{dumpling-cond}.

Then there exists a $k$-set $U$ of discs $D(z_n,r_n)$ satisfying \eqref{logsum4} and a constant $C>0$ depending only on $\alpha, b,k, j$ such that for all $z\not\in U$,
we have the following estimate (where $r=|z|$):
    \begin{equation}\label{k-thm-assertion-D}
    \left|\frac{f^{(k)}(z)}{f^{(j)}(z)}\right|
    \leq C\left(\frac{T(s(r),f)-\log (1-r)}{(1-r)^2}+W(r)\right)^{k-j},
    \end{equation}
where
	$$
	W(r)=\frac{n_j(s(r))}{k(1-r)}\log^+n_j(s(r))\log^\alpha\frac{1}{1-r}
	$$
and $n_j(t)$ is the number of zeros and poles of $f^{(j)}$ in the disc $\{\zeta:|\zeta|\leq t\}$, counting multiplicities.
    \end{corollary}

We note that \cite[Theorem~4]{Chyzhykov0} is in the spirit of Corollary~\ref{k-thm-D} but
the estimate in it is obtained by other means and is valid outside of a fixed exceptional set that corresponds to the case $k(x)=x$. Meanwhile, the estimate in \eqref{k-thm-assertion-D} depends on the
size of the exceptional set determined by~$k$.

\bigskip
\noindent
\emph{Proof of Corollary~\ref{k-thm-D}.}
Let $\{a_m\}$ denote the sequence of zeros and poles of $f^{(j)}$ listed according to
multiplicity and ordered by increasing modulus. We follow the proof of 
\cite[Theorem~3.1]{Chyzhykov}, which in turn is reminiscent to the 
proof of \cite[Theorem~3]{Gundersen}. In the unit disc case one just
has to be extra careful so that the Cartan discs still remain in $\D$. Set
    $$
    \mu_\nu=n_j(1-b^{\nu+2})\quad\textnormal{and}\quad
    d_\nu=\frac{k(b^\nu)}{\nu^{\alpha}(-\log b)^\alpha},
    $$
where $\nu\geq \nu_0$ is an integer. We make this choice for the constants $d_\nu$
independently on whether $k(x)=x$ or not. We also suppose that $z$ is confined to the annulus $\A_\nu=\{\zeta:1-b^\nu\leq |\zeta|< 1-b^{\nu+1}\}$, which is equivalent to saying that
	$$
	\nu(-\log b)\leq-\log(1-r)< (\nu+1)(-\log b),\quad |z|=r.
	$$

We now make use of Cartan's lemma as in \cite{Chyzhykov}. Indeed, if $z$ lies outside of Cartan's discs, and if $\nu_0$ is assumed to be large
enough so that $\log \mu_\nu\geq 1$ for all $\nu\geq\nu_0$, then the following analogue of (6.5)--(6.7) in \cite{Chyzhykov} holds:
    \begin{eqnarray*} \label{ineq 1}
    \sum_{|a_m|\leq s(r)}\frac{1}{|z-a_m|}
    &\leq& \sum_{m=1}^{\mu_\nu}\frac{1}{|z-b_m|}<\frac{\mu_\nu \nu^{\alpha}(-\log b)^\alpha}{k(b^\nu)}\sum_{m=1}^{\mu_\nu}\frac{1}{m}\\
    &\leq& \frac{\mu_\nu \nu^{\alpha}(-\log b)^\alpha}{k(1-r)}(1+\log \mu_\nu)\\
    &\leq& \frac{\mu_\nu(1+\log \mu_\nu)}{k(1-r)}\log^\alpha\frac{1}{1-r}\\\
    &\leq& 2\frac{n_j(1-b^2(1-r))}{k(1-r)}\log n_j(1-b^2(1-r))\log^\alpha\frac{1}{1-r}.
    \end{eqnarray*}
The assertion \eqref{k-thm-assertion-D} follows from this and from \cite[(6.1)]{Chyzhykov}.
Thus \eqref{k-thm-assertion-D} holds for a fixed $\nu\geq \nu_0$ when $z$ lies outside of Cartan's discs.

Similarly as in \cite{Chyzhykov}, we now consider all $\nu\geq\nu_0$,
where $\nu_0\geq 1$ is large enough.
Let $A_{\nu,1},\ldots,A_{\nu,l_\nu}$ denote precisely those Cartan
discs that intersect the annulus $\A_\nu.$
We mention that for some $\nu$ there might be no discs of this type.
It follows from Cartan's lemma that for each $\nu$ the total sum
of the diameters of the discs $A_{\nu,1},\ldots,A_{\nu,l_\nu}$
cannot exceed $4d_\nu$.
Let $\zeta_{\nu,i}$ and $\xi_{\nu,i}$ denote the center and radius of
the disc $A_{\nu,i}$, respectively. It is easy to see that 
	$$
	|\zeta_{\nu,i}|+\xi_{\nu,i}\leq 1-b^{\nu+1}+2\xi_{\nu,i}
	\leq 1-b^{\nu+1}+4d_\nu.
	$$
If $k$ is not the identity mapping, then from \eqref{l-zero}, we infer
	$$
	\lim_{\nu\to\infty}\frac{k(b^\nu)}{b^\nu}
	=\lim_{\nu\to\infty}\frac{k(1-(1-b^\nu))}{1-(1-b^\nu)}
	=0.
	$$
Thus we may choose $\nu_0$ large enough so that $4d_\nu<b^{\nu+1}$
for all $\nu\geq \nu_0$. This inequality holds also in the case when
$k(x)=x$, but possibly for a different $\nu_0$. Now $|\zeta_{\nu,i}|+\xi_{\nu,i}<1$ for all $\nu\geq\nu_0$, which means that $A_{\nu,1},\ldots,A_{\nu,l_\nu}\subset\D$ for all $\nu\geq\nu_0$. Finally,
we choose a larger $\nu_0$, if necessary, so that $b^{\nu+1}<1-b^\nu$
for all $\nu\geq\nu_0$. Then $1-b^\nu>4d_\nu$ for all $\nu\geq\nu_0$,
which means that the origin lies outside of the discs $A_{\nu,1},\ldots,A_{\nu,l_\nu}$ for all $\nu\geq\nu_0$.

Now let $D_n=D(z_n,r_n)$ denote the sequence $\{D_n\}$ of all the
discs $A_{\nu,i}$, where $\nu\geq\nu_0$ and $1\leq i\leq l_\nu$.
We have proved that \eqref{k-thm-assertion-D} holds if $z$ lies
outside of the discs $D_n$ and $\{\zeta:|\zeta|\leq 1-b^{\nu_0}\}$.
Moreover, since $2d_\nu/b^\nu\to 0$ as $\nu\to\infty$, we use
\eqref{dumpling-cond} for $\gamma=b$ to find a constant $M>0$ 
that depends on
$b,\nu_0$ and on the constant $\tau$ in \eqref{dumpling-cond} such that
	\begin{eqnarray*}
	\sum_{n=1}^\infty\frac{r_n}{k(1-|z_n|)}
	&=& \sum_{\nu=\nu_0}^\infty \sum_{D_n\cap \A_\nu\neq\emptyset}\frac{r_n}{k(1-|z_n|)}\leq\sum_{\nu=\nu_0}^\infty\frac{2d_\nu}{k(b^{\nu+1}-2d_\nu)}\\
	&=&\sum_{\nu=\nu_0}^\infty\frac{2d_\nu}{k(b^{\nu}(b-2d_\nu/b^{\nu}))}
	\leq \sum_{\nu=\nu_0}^\infty\frac{M}{\nu^\alpha}
	<\infty.
	\end{eqnarray*}
This completes the proof. \hfill$\Box$

\subsection{Logarithmic differences and $q$-differences}

A difference counterpart to Gundersen's pointwise estimates for logarithmic derivatives due to Chiang and Feng, see \cite[Theorem~8.2]{C-F}, is also well-known. The estimate holds for all $z$ such that $|z|$ lies outside of an exceptional set of finite logarithmic measure. For a $K$-version of it, all we need is to change the estimate (8.6) in \cite{C-F} by the  reasoning used in proving Theorem~\ref{k-thm}, 
and use the same constants $d_\nu$. We state the result
formally but omit the proof.

\begin{corollary} \label{kk-thm}
Let $f\not\equiv 0$ be a meromorphic function in $\C$, 
and let $\alpha >1$ and $c\in\C$ be given constants. Suppose that 
$K:[0,\infty)\to [0,\infty)$ is as in Corollary~\ref{k-thm}.
Then there exists a $K$-set $U$ of discs $D(z_n,r_n)$ satisfying \eqref{logsum} and a constant $C>0$ depending only on $\alpha$ and $c$ such that for all $z\not\in U$, we have the following estimate (where $r=|z|$):
    \begin{equation}\label{difference}
	\left|\log\left|\frac{f(z+c)}{f(z)}\right|\right|
	\leq C\left(\frac{T(\alpha r,f)}{r}
	+n(\alpha r)\log^+n(\alpha r)\frac{\log^{\alpha}r}{K(r)}\right).
	\end{equation}
Here $n(t)$ is the number of zeros and poles of $f$ in the disc $\{\zeta:|\zeta|\leq t\}$, counting multiplicities.
    \end{corollary}

A $q$-difference counterpart to Gundersen's pointwise estimates
was discovered by Wen and Ye in \cite{WY}. For a $K$-version of it, all we need is to change the estimate (3.9) in \cite{WY} by the  reasoning used in proving Theorem~\ref{k-thm}, and use the same constants $d_\nu$, and
then \cite[Lemma~3.4]{WY} yields the following analogue of \eqref{difference}: 
	\begin{eqnarray*}
	\log\frac{f(qz)}{f(z)} 
	&=& (q-1)z\frac{f'(z)}{f(z)}
	+O\left(\frac{|q-1|^2m(r^2)}{r^2}\right)
	+O\left(\frac{|q-1|^2 n(r^2)}{r^2}\right)\\
	&& +O\left(rn(\alpha r) \log^
	+n_j(\alpha r)\frac{\log^\alpha r}{K(r)}\right)+O(1),
	\end{eqnarray*}
where $m(r)=m(r,f)+m(r,1/f)$. The details are omitted.

\subsection{Exceptional sets}\label{except-sets}

Suppose that $K:(0,\infty)\to (0,\infty)$ is one of the following
three types of functions: (1) $K(x)\equiv 1$, (2) $K(x)\equiv x$,
or (3) $1\leq K(x)\leq x$ and $K$ is strictly increasing, continuous, concave function satisfying the doubling condition \eqref{double}.
Suppose further that $E\subset [1,\infty)$ satisfies $\int_E\frac{dx}{K(x)}<\infty$. The cases (1) and (2) correspond to finite linear measure
and finite logarithmic measure, respectively. It is clear that in all three cases,
	$$
	\lim_{r\to\infty}\frac{K(r)}{r}\int_{E\cap [r,\infty)}\frac{dx}{K(x)}=0.
	$$
Define the $K$-density of the set $E$ relative to a given function $\varepsilon:(0,\infty)\to (0,1]$ by
	$$
	\delta(K,\varepsilon)(E)=
	\limsup_{r\to\infty}\frac{\frac{K(r)}{r}\int_{E\cap [r,\infty)}\frac{dx}{K(x)}}{\varepsilon(r)}.
	$$
If $\delta(K,\varepsilon)(E)\in (0,\infty)$, then it is clear that $\varepsilon(r)$ tends to zero (along a sequence of values $r$), 
but not arbitrarily fast.

Special cases of Lemma~\ref{tech-lemma} below allow us to avoid exceptional sets of finite linear measure ($K(x)\equiv 1$) or of finite logarithmic measure ($K(x)\equiv x$). The proof
is an easy modification of the original results, see \cite[Lemma~C]{Bank} and \cite[Lemma~5]{Gundersen1}, but the result improves the existing estimates in the sense that a constant $\alpha>1$ is replaced by a function $\alpha(r)>1$ that tends to 1. Indeed, in the original results the function $\veps(r)$ is a constant function $\veps(r)\equiv \varepsilon=\alpha-1$, in 
which case $\delta(K,\varepsilon)(E)=0$ trivially holds. The $K$-density
introduced above controls how fast $\alpha(r)$ tends to 1.

\begin{lemma}\label{tech-lemma}
Let $g(r)$ and $h(r)$ be non-decreasing functions on $(0,\infty)$, and 
suppose that $g(r)\leq h(r)$ for all $r\in (0,\infty)\setminus E$, 
where $\delta(K,\varepsilon)(E)<1/\alpha$, and $K$ and $\varepsilon$ are as above, while $\alpha>0$
is the constant from \eqref{double}. Denote
$\alpha(r)=1+\varepsilon(r)$. Then there exists an $R\geq 1$ such that
$g(r)\leq h(\alpha(r)r)$ for all $r\geq R$.
\end{lemma}

\begin{proof}
There exists a constant $R\geq 1$ such that the doubling condition
\eqref{double} holds for all $x\geq R$, and that
	$$
	\varepsilon(r)>\alpha\frac{K(r)}{r}\int_{E\cap [r,\infty)}\frac{dx}{K(x)},\quad r\geq R.
	$$
Thus, if $r\geq R$,
	$$
	\int_r^{\alpha(r)r}\frac{dx}{K(x)}\geq \frac{(\alpha(r)-1)r}{K(\alpha(r)r)}\geq \frac{\veps(r)r}{K(2r)}\geq \frac{\veps(r)r}{\alpha K(r)}>\int_{E\cap [r,\infty)}\frac{dx}{K(x)}.
	$$
It follows that $[r,\alpha(r)r]\setminus E\neq\emptyset$ for all
$r\geq R$. So, if $r\geq R$ is arbitrary, we may choose $t\in [r,\alpha(r)r]\setminus E$ such that
	$$
	g(r)\leq g(t)\leq h(t)\leq h(\alpha(r)r)
	$$
by the monotonicity of $g$ and $h$.
\end{proof}	
	
Next we suppose that $k:(0,1)\to (0,1)$ is either the identity
mapping or a strictly increasing, continuous and convex function 
satisfying \eqref{dumpling-cond} and \eqref{l-zero}.
Suppose further that $E\subset [0,1)$ satisfies $\int_E\frac{dx}{k(1-x)}<\infty$. Define the $k$-density of the set $E$ relative to a given 
function $b:(0,1)\to (0,1)$ by
	$$
	\delta(k,b)(E)=
	\limsup_{r\to 1^-}\frac{\frac{k(1-r)}{1-r}\int_{E\cap [r,1)}\frac{dx}{k(1-x)}}{1-b(r)}.
	$$
We state without a proof the following result which reduces to
\cite[Lemma~C]{Bank} in the case when $k(x)=x$ and when $b(r)=b$ is a
constant function.

\begin{lemma}\label{tech-lemma2}
Let $g(r)$ and $h(r)$ be non-decreasing functions on $(0,1)$, and 
suppose that $g(r)\leq h(r)$ for all $r\in (0,1)\setminus E$, 
where $\delta(k,b)(E)<1$, and $k$ and $b$ are as above. Denote
$s(r)=1-b(r)(1-r)$. Then there exists an $R\geq 0$ such that
$g(r)\leq h(s(r))$ for all $r\geq R$.
\end{lemma}

\bigskip

\noindent
\emph{J.~Ding}\\
\textsc{Taiyuan University of Technology,
Department of Mathematics,
Yingze West Street, No.~79, Taiyuan 030024, China}\\
\texttt{e-mail:dingjie@tyut.edu.cn}

\medskip
\noindent
\emph{J.~Heittokangas}\\
\textsc{University of Eastern Finland, Department of Physics and Mathematics,
P.O.~Box 111, 80101 Joensuu, Finland}\\
\textsc{Taiyuan University of Technology,
Department of Mathematics,
Yingze West Street, No.~79, Taiyuan 030024, China}\\
\texttt{email:janne.heittokangas@uef.fi}

\medskip
\noindent
\emph{Z.-T.~Wen}\\
\textsc{Shantou University, Department of Mathematics,
Daxue Road No.~243, Shantou 515063, China}\\
\textsc{Taiyuan University of Technology,
Department of Mathematics,
Yingze West Street, No.~79, Taiyuan 030024, China}\\
\texttt{e-mail:zhtwen@stu.edu}

\end{document}